\documentclass[hidelinks, 12pt, oneside, reqno]{amsart}
\usepackage{dsliheader}
\usepackage{cite}
\newtheoremstyle{named}{}{}{\itshape}{}{\bfseries}{.}{.5em}{\thmnote{#3}}
\theoremstyle{named}

\usepackage{geometry}
\title[Closed $G_2$-Structures with Negative Ricci Curvature]{Closed $G_2$-Structures with Negative Ricci Curvature}

\author[Payne]{Alec Payne}
\thanks{North Carolina State University, SAS 4210, 2311 Stinson Drive, Raleigh, NC 27607, USA}
\date{}
\begin{document}

\begin{abstract}
We study existence problems for closed $G_2$-structures with negative Ricci curvature, and we prove the $G_2$-Goldberg conjecture for noncompact manifolds. We first show that no closed manifold admits a closed $G_2$-structure with negative Ricci curvature. In the noncompact setting, we show that no complete manifold admits a closed $G_2$-structure with Ricci curvature pinched sufficiently close to a negative constant. As a consequence, an Einstein closed $G_2$-structure on a complete manifold must be torsion-free. In addition, when the Einstein metric is incomplete, we find restrictions on lengths of geodesics.

\end{abstract}
\maketitle

\vspace{-.3in}

\section{Introduction}


When searching for Riemannian manifolds with holonomy $G_2$, the first step is usually to find a closed $G_2$-structure, i.e.\ a $G_2$-structure defined by a closed $3$-form. Despite their importance, relatively little is known about closed $G_2$-structures. The only available methods for constructing them on compact manifolds are perturbative techniques or symmetry and dimension reduction~\cite{FinoRafferoResultsonClosed21}. 

There are no known restrictions on the topology of a $7$-manifold $M$ admitting a closed $G_2$-structure, beyond orientability and spinnability of $M$. The most well-known, yet still very weak, geometric restriction on a closed $G_2$-structure is that it has nonpositive scalar curvature. Moreover, it is scalar flat if and only if the holonomy group is contained in $G_2$ (see~\eqref{equation R = -|T|^2}). This suggests that nonpositive curvature restrictions on closed $G_2$-structures are geometrically significant.

It is important to understand the behavior of closed $G_2$-structures under general curvature restrictions. For instance, one may hope to construct new $G_2$-holonomy metrics using the Laplacian flow, a geometric flow of closed $G_2$-structures, but a major obstacle is that there are no known, nontrivial curvature restrictions preserved by this flow.

In this paper, we rule out the existence of closed $G_2$-structures with negative Ricci curvature in some general settings. Our first result says that closed $7$-manifolds with a closed $G_2$-structure may not have negative Ricci curvature.

\begin{thm}\label{theorem compact negative Ric}
    There does not exist a closed $7$-manifold with a closed $G_2$-structure $\vp$ such that $\Ric(g_{\vp})<0$.
\end{thm}

This result may be surprising because negative Ricci curvature is a very weak geometric condition. Even though every $n$-manifold, $n\geq 3$, admits a complete metric with negative Ricci curvature~\cite{lohkamp1994metrics}, it turns out that closed $G_2$-structures are often incompatible, in a global sense, with negative Ricci curvature metrics.

Our proof technique for Theorem \ref{theorem compact negative Ric} is to formulate an integral identity for closed $G_2$-structures involving the Ricci curvature (see Lemma \ref{lemma integral identity Ric T2}). By noting that a certain term in this identity has a sign, we find a contradiction if the metric has negative Ricci curvature. Our argument for Theorem \ref{theorem compact negative Ric} also rules out $\Ric(g_{\vp})>0$, but this is trivial because $g_{\vp}$ has nonpositive scalar curvature when $\vp$ is closed.

Theorem \ref{theorem compact negative Ric} is sharp since the condition $\Ric(g_{\vp})<0$ may not be relaxed to $\Ric(g_{\vp})\leq 0$. Indeed, there exist closed $7$-manifolds admitting closed $G_2$-structures of nonpositive Ricci curvature. One source of examples comes from extremally Ricci pinched (ERP) $G_2$-structures, which have nonpositive Ricci curvature. Many examples of ERP, closed $G_2$-structures have been found on compact manifolds (see ~\cite{BryantRemarks,Lauret17, BallQuadratic, KathLauret, FinoRafferoEternal}). The techniques of Theorem \ref{theorem compact negative Ric} tell us that a closed $G_2$-structure on a closed manifold with nonpositive Ricci curvature must satisfy a special condition on its torsion: $d(\tau^3)=0$. For comparison, ERP $G_2$-structures on closed manifolds satisfy $\tau^3 = 0$~\cite[(4.53)]{BryantRemarks}. To the author's knowledge, there are no examples of closed manifolds admitting non-ERP closed $G_2$-structures which have nonpositive Ricci curvature and are not Ricci flat. 

Next, we extend Theorem \ref{theorem compact negative Ric} to the case of noncompact manifolds with negatively pinched Ricci curvature, and we resolve the $G_2$-Goldberg conjecture in the noncompact setting. The $G_2$-Goldberg conjecture states that an Einstein closed $G_2$-structure $\vp$ must be torsion-free, i.e.\ $\mathrm{Hol}(g_{\vp})\subseteq G_2$ and, in particular, $g_{\vp}$ is Ricci flat. This conjecture was first considered in various guises in the physics literature (see~\cite{GibbonsPagePope90, CveticGibbonsPope}) and was resolved for closed manifolds in 2003 by Cleyton--Ivanov~\cite{CleytonIvanovGeometryClosed} and Bryant~\cite[(4.40)]{BryantRemarks}. The $G_2$-Goldberg conjecture is the $G_2$-geometry analogue of the Goldberg conjecture in K\"ahler geometry, which states that an almost K\"ahler, Einstein manifold\footnote{This conjecture was originally stated with the extra assumption of compactness.} must be K\"ahler~\cite{goldberg1969integrability}. The analogy between the $G_2$-Goldberg conjecture and the Goldberg conjecture is that closed and torsion-free $G_2$-structures are akin to almost K\"ahler and K\"ahler structures, respectively. The Goldberg conjecture is known to be false for some noncompact almost K\"ahler manifolds but the question remains open in the compact setting~\cite{apostolov2001, apostolov2003}. In sharp contrast to the K\"ahler setting, we find that complete Einstein closed $G_2$-structures do not exist, even on noncompact manifolds, as a consequence of the following theorem.

\begin{thm}\label{theorem Ricci pinched Einstein}
    There does not exist a noncompact $7$-manifold with a closed $G_2$-structure $\vp$ and complete metric $g_{\vp}$ satisfying
    \begin{equation}\label{equation almost Einstein Ric bound}
    -k_2 g_{\vp} \leq \Ric(g_{\vp}) \leq -k_1 g_{\vp},\end{equation}
    for $0 < k_1 \leq k_2$ such that 
    \begin{equation}\label{equation pinching of k1 and k2}\frac{k_1}{k_2} > \left(\frac{7}{18}\right)^{\frac{1}{4}} \approx .79.\end{equation}
\end{thm}

Theorem \ref{theorem Ricci pinched Einstein} implies Corollary \ref{corollary complete noncompact Einstein}, which is the $G_2$-Goldberg conjecture without a compactness assumption. This follows since closed $G_2$-structures have nonpositive scalar curvature and are torsion-free exactly when the scalar curvature is identically zero.


\begin{cor}\label{corollary complete noncompact Einstein}
If $(M^7,\vp)$ is a closed $G_2$-structure with a complete Einstein metric $g_{\vp}$, then $(M^7, \vp)$ is torsion-free. In particular, $g_{\vp}$ is Ricci flat.
\end{cor}


The $G_2$-Goldberg conjecture for noncompact manifolds has been posed as an open problem in~\cite[Question 3.7]{Lauret17} and~\cite[Remark 3.2]{FinoRafferoResultsonClosed21}. It has previously been solved in the case of solvmanifolds~\cite{FernandezFinoManeroEinsteinSolv12}, in the case of warped products~\cite{ManeroUgarteEinsteinWarped19}, and under the additional assumption of ``$\star$-Einstein''~\cite[Prop.\ 5.7]{CleytonIvanovGeometryClosed} (cf.\ the analogous result for the Goldberg conjecture ~\cite{OguroSekigawa98}). In general, a local version of the conjecture remains open: there are no known examples of non-torsion-free, Einstein closed $G_2$-structures defined on a ball in $\bR^7$~\cite[Remark 12]{BryantRemarks}. 

Any proof of the $G_2$-Goldberg conjecture must go beyond a pointwise analysis of Einstein closed $G_2$-structures because it is possible for a closed $G_2$-structure to have all its Ricci eigenvalues be negative and identical at a single point. Our technique to prove Theorem \ref{theorem Ricci pinched Einstein} is to derive an estimate for the volume growth (Lemma \ref{lemma volume growth Ricci pinched}) which is strictly faster than the volume growth given by the Bishop--Gromov inequality. This volume growth estimate is found by integrating a special identity (Lemma \ref{lemma integral identity Ric T2}) on balls of arbitrary radius. Some technicalities from geometric measure theory are needed to handle the non-smooth boundaries of large balls. Our argument relies on exact constants in certain $G_2$ identities, some of which are not well-known, so we provide a largely self-contained derivation of these in Section \ref{section preliminaries}.

Our argument for Theorem \ref{theorem Ricci pinched Einstein} also gives restrictions on closed $G_2$-structures with incomplete metrics satisfying the pinching conditions~\eqref{equation almost Einstein Ric bound} and~\eqref{equation pinching of k1 and k2}. We find that there are upper bounds on lengths of geodesics under these restrictions. 

\begin{thm}\label{theorem incomplete Einstein}
    Let $(M, \vp)$ be a closed $G_2$-structure with a possibly incomplete metric $g_{\vp}$ satisfying~\eqref{equation almost Einstein Ric bound} and~\eqref{equation pinching of k1 and k2}. Then, there exists $C = C(k_1, k_2) >0$ such that for each $p \in M$, $\exp_p: T_p M \to \bR$ is undefined for some $v \in T_p M$ with $|v|_{g_{\vp}} = C/\sqrt{|k_2|}$. 
    
    In particular, if $g_{\vp}$ is Einstein, then $C \approx 2.05$.
\end{thm}
    


The assumption that the $G_2$-structure is closed is necessary in all of our results. For example, every nearly parallel $G_2$-structure produces a positive Einstein metric, yet nearly parallel $G_2$-structures are not closed~\cite[Prop.\ 3.10]{FriedrichNearlyParallel}. There even exists a homogeneous $G_2$-structure on a noncompact manifold which is locally conformally equivalent to a closed $G_2$-structure and which produces a non-Ricci flat Einstein metric~\cite{FinoRafferoEinsteinconformal}. 


It is unknown if the pinching restriction~\eqref{equation pinching of k1 and k2} is sharp. One example of a closed $G_2$-structure with negatively pinched Ricci curvature, due to Ball, is defined on $\Lambda^2_{+}\mathbb{C}\mathbb{H}^2$ and produces an incomplete metric satisfying ~\eqref{equation almost Einstein Ric bound} with $k_1/k_2 = 1/3$~\cite[\S 5.1.3]{BallConfomallyflat}. In addition, there exist complete, $\mathrm{Sp}(2)$-invariant closed $G_2$-structures on $\Lambda^2_{-}\bS^4$ which have Ricci curvature negative and bounded away from zero outside a compact set. These examples have volume growth $\exp(r^2)$ and may be found by solving the ODE corresponding to the closed condition for an $\mathrm{Sp(2)}$-invariant $G_2$-structure~\cite[Prop.\ 3.48]{HaskinsNordstrom21}.

\subsection*{Acknowledgments}
The author would like to thank Mark Haskins for helpful discussions and for suggesting some of the problems addressed by this paper. The author would also like to thank Robert Bryant, Shubham Dwivedi, and Gavin Ball for helpful comments and discussions.
 
The author would also like to thank the Simons Foundation for its support under the Simons Collaboration on Special Holonomy in Geometry, Analysis and Physics, grant \#488620.

\section{Preliminaries}\label{section preliminaries}

\subsection{Preliminaries on \texorpdfstring{$G_2$}{G2}-structures}

In this section, we will recall some relevant facts from $G_2$-geometry, and we refer to~\cite{BryantRemarks, JoyceBook00, KarigiannisIntro20, KarigiannisFlows09} for general background. We take care to state our conventions and definitions explicitly, since particular constants in $G_2$ identities are important for our results. We prove a few identities which are not generally well-known in order to verify constants. 

We begin by stating our conventions for coordinate expressions of differential forms. If $M$ is a smooth manifold and $\alpha \in \Om^k(M)$, then in local coordinates,
$$\alpha = \frac{1}{k!} \alpha_{i_1, i_2, \cdots, i_k} dx^{i_1} \we \cdots \we dx^{i_k} = \sum_{i_1 < \cdots < i_k} \alpha_{i_1, \cdots, i_k} \, dx^{i_1} \we \cdots \we dx^{i_k},$$
where $\alpha_{i_1, \cdots, i_k}$ is totally antisymmetric in $i_1, \dots, i_k$.

Given an orientation and Riemannian metric $g$ on $M$, the inner product of two $k$-forms $\alpha$ and $\beta$ is given in local coordinates by
\begin{equation}\label{equation inner product of forms}
    \langle\alpha, \beta\rangle \Vol= \al \we \ast \be = \frac{1}{k!} \alpha_{i_1, \cdots, i_k} \beta^{i_1, \cdots, i_k} \Vol,\end{equation}
and so $|\alpha|^2= \langle \alpha, \alpha\rangle = \frac{1}{k!}|\alpha|_g^2$, where $|\!\cdot\!|_g^2$ denotes the usual inner product on tensors induced by $g$. 

In this paper, $M$ will denote a connected, smooth $7$-manifold. A $G_2$-structure on $M$ is equivalent to the choice of a $3$-form $\vp \in \Omega^3(M)$ which equips each tangent space of $M$ with the structure of a triple product arising from a $7$-dimensional cross product. We then say that $(M, \vp)$ is a $G_2$-structure, where $\vp$ is the associated $3$-form. Produced by any $G_2$-structure $(M, \vp)$ is an orientation and metric $g_{\vp}$, defined by
\begin{equation}\label{equation metric and orientation formula}g_{\vp}(X,Y) \Vol = \frac{1}{6}\io_X \vp \we \io_Y \vp \we \vp,\end{equation}
where $X$ and $Y$ are vector fields on $M$ and $\Vol$ is the $7$-dimensional volume form on $M$ induced by $g_{\vp}$. We will sometimes suppress the subscript in $g_{\vp}$ and simply write $g$ for brevity. 

The metric and orientation produced by $\vp$ induce a Hodge star. Using the Hodge star, we define $\psi := \ast \vp$. We then define the torsion of a $G_2$-structure $\vp$ as the two-tensor $T = T_{ij}$ given by
$$T_{ij} = \frac{1}{24}\na_i \vp_{abc} \psi_j^{\;\,abc}.$$
A classical fact in $G_2$-geometry is that $\Hol(g_{\vp})\subseteq G_2$ if and only if the $G_2$-structure $(M,\vp)$ is torsion-free, i.e.\ $T=0$~\cite{FernandezGray82}. Moreover, torsion-free $G_2$-structures induce Ricci flat metrics, so $\Ric(g_{\vp})=0$.

For each $p\in M$, the $G_2$-structure $\vp$ induces a decomposition of $\La^k(T_p^* M)$ into $G_2$-irreducible components, orthogonal with respect to $g_{\vp}$. There is a corresponding splitting of $\Om^k(M)$ into $g_{\vp}$-orthogonal subbundles. In particular,  
\begin{align}
\Om^2(M) &= \Om^2_7(M) \oplus \Om^2_{14}(M),\\
\Om^3(M) &= \Om^3_1(M) \oplus \Om^3_7(M) \oplus \Om^3_{27}(M),
\end{align}
where $\Om^k_{\ell}(M)$ denotes a rank $\ell$ subbundle of $\Om^k(M)$. 

If $\eta \in \Om^2_{14}(M)$, then
\begin{align}
\eta \we \vp &= -\ast \eta,\label{equation om214 condition}\\
|\eta^2|^2 &= |\eta|^4,\label{equation tau squared}\\
|\eta^3|^2 &\leq \frac{2}{3}|\eta|^6.\label{equation tau cubed}
\end{align}
The identity~\eqref{equation om214 condition} comes from~\cite[(2.14)]{BryantRemarks}. We note that the sign in~\eqref{equation om214 condition} depends on our orientation convention. Identities~\eqref{equation tau squared} and ~\eqref{equation tau cubed} come from~\cite[(2.21), (2.22)]{BryantRemarks}.







We say that $(M, \vp)$ is a closed $G_2$-structure if $d\vp = 0$. 
If $\vp$ is a closed $G_2$-structure, then its torsion is a $2$-form in $\Om^2_{14}(M)$~\cite[Prop.\ 1]{BryantRemarks} (see also~\cite[Section 2.1]{LotayWeiLaplacianFlow}). In particular,
\begin{equation}\label{equation antisymmetry of T}
    T_{ij} = -T_{ji}.
\end{equation}
We define a two-form $\tau$ such that $\tau_{ij} = -2T_{ij}$, and we have that
\begin{equation}\tau = \frac{1}{2} \tau_{ij} dx^i \wedge dx^j \in \Om^2_{14}(M).\end{equation}

\noindent A crucial fact about closed $G_2$-structures is that the scalar curvature $R$ satisfies
\begin{equation}\label{equation R = -|T|^2}R = -|T|_g^2  = g^{ij}T_{il}T^{l}_{\;j} =  - T_{ij}T^{ij}= -\frac{1}{2}|\tau|^2.\end{equation}

Let $S^2 T^*\!M$ be the space of symmetric two-tensors on $M$. As in~\cite[(2.17)]{BryantRemarks}, we may define a map $i_{\vp}: S^2 T^*\!M \to \Om^3(M)$ by
\begin{align}\label{equation iophi definition}
    i_{\vp}(h)_{ijk} &= \frac{1}{2} h_i^l\vp_{ljk} dx^i \wedge dx^j \wedge dx^k\nonumber\\
    &= \frac{1}{6} \big(h_i^l \vp_{ljk} - h_j^l \vp_{lik} - h_k^l \vp_{lji}\big) dx^i \wedge dx^j \wedge dx^k,
\end{align}
where $h = h_{ij}$ is a symmetric two-tensor. The map $i_{\vp}$ is a bijection from trace-free symmetric two-tensors to $\Om^3_{27}(M)$ and so $i_{\vp}(S^2_0 T^*\!M) = \Om^3_{27}(M)$. 

Definition~\eqref{equation iophi definition}, matches the definitions of Karigiannis~\cite{KarigiannisFlows09} and Lotay--Wei~\cite{LotayWeiLaplacianFlow} but differs from Bryant's definition by a factor of $\frac{1}{2}$~\cite[(2.17)]{BryantRemarks}. With our definition, we have that $i_{\vp}(g_{\vp}) = 3 \vp$.

We will now calculate $i_{\vp}(\Ric)$ in our conventions, based on work of Bryant~\cite[(4.39)]{BryantRemarks}. 
\begin{lem}[\!\mbox{\cite[(4.39)]{BryantRemarks}}]\label{lemma iophi of Ric}
Let $(M, \vp)$ be a closed $G_2$-structure. Then,
    \begin{equation}\label{equation iophi of Ric}
        i_{\vp}(\Ric) = -d\tau  + \frac{1}{2}\!\ast\!(\tau \we \tau).
    \end{equation}
\end{lem}
\begin{proof}
For a closed $G_2$-structure~\cite[(4.39)]{BryantRemarks}, Bryant found that
\begin{equation}\label{equation Bryant's i(Ric)}
i_{\vp}(\Ric^0) = \frac{3}{14}|\tau|^2 \vp  - d\tau  + \frac{1}{2}\!
\ast\!(\tau \we \tau),\end{equation}
which we have adjusted to match our definition of $i_{\vp}$. We now find that
\begin{align}
    i_{\vp}(\Ric) &= i_{\vp}\Big(\Ric^0 + \frac{R}{7} g_{\vp}\Big)\nonumber\\
    &= i_{\vp}(\Ric^0) + \frac{3R}{7}\vp\nonumber\\
\intertext{Using~\eqref{equation R = -|T|^2},}
    &= i_{\vp}(\Ric^0) - \frac{3}{14}|\tau|^2 \vp\nonumber\\
\intertext{Applying~\eqref{equation Bryant's i(Ric)},}
    &= -d\tau  + \frac{1}{2}\!\ast\!(\tau \we \tau).
\end{align}
\end{proof}
We will make use of Lemma \ref{lemma iophi of Ric} using the following lemma, which is well-known and can be found in~\cite[Prop.\ 2.15]{KarigiannisFlows09}. It follows from a standard tensor computation of $i_{\vp}(U) \we \ast i_{\vp}(V)$ using~\eqref{equation iophi definition} and then an application of known $G_2$-identities\cite[(2.6), (2.8)]{BryantRemarks}. 
\begin{lem}[\!\mbox{\cite[Prop.\ 2.15]{KarigiannisFlows09}}]\label{lemma iophi inner product}
   Let $(M, \vp)$ be a $G_2$-structure. If $U = U_{ij}$ and $V = V_{kl}$ are symmetric two-tensors, then
    $$i_{\vp}(U) \we \ast i_{\vp}(V) = \big(g^{ij}U_{ij} g^{kl}V_{kl} + 2U^{ij}V_{ij}\big)\!\Vol.$$
\end{lem}

We end this section with a useful identity for closed $G_2$-structures in Lemma \ref{lemma integral identity Ric T2}, which is the linchpin of the proofs of Theorems \ref{theorem compact negative Ric} and \ref{theorem Ricci pinched Einstein}. We note that the same identity appeared in a different form in ~\cite[Cor.\ 7.6]{CleytonIvanovCurvatureDecomp08}, but we give a simpler proof here. One could also prove Lemma \ref{lemma integral identity Ric T2} by starting with Bryant's expression for $d(\tau^3)$~\cite[(4.45)]{BryantRemarks}, but this approach is more difficult.
\begin{lem}[\!\mbox{\cite[Cor.\ 7.6]{CleytonIvanovCurvatureDecomp08}}] \label{lemma integral identity Ric T2}
Let $(M, \vp)$ be a closed $G_2$-structure. Then,
\begin{equation}\label{equation identity of Lemma 2.4}
24 R_{ij}T^{il}T_{l}^{\;j} = \ast d(\tau^3),\end{equation}
where $R_{ij}$ is the Ricci curvature of $g_{\vp}$. 

In particular, if $M$ is a closed manifold,
$$\int_M R_{ij}T^{il}T_{l}^{\;j} = 0.$$
\end{lem}
\begin{proof}
First, we note that for a closed $G_2$-structure, $d\tau = \De_{\vp}\vp$~\cite[(2.17)]{LotayWeiLaplacianFlow}. Then, by a computation of Lotay-Wei~\cite[(3.2)--(3.4)]{LotayWeiLaplacianFlow},
\begin{equation}\label{equation Lotay wei computation of io phi}
    i_{\vp}\Big(-\Ric + \frac{1}{3}Rg - 2 T^2\Big) = \De_{\vp}\vp = d\tau,
\end{equation} 
where $T^2$ is the symmetric two-tensor given by $(T^2)_{ij} = T_i^{\;l}T_{lj}$. We note that~\eqref{equation Lotay wei computation of io phi} could alternatively be derived from~\cite[(4.37)]{BryantRemarks}. Rearranging~\eqref{equation Lotay wei computation of io phi}, we find that
\begin{align}\label{equation iphi T2 computation}
    i_{\vp}(T^2) &= - \frac{1}{2} d\tau + \frac{1}{2}i_{\vp}\Big(-\Ric + \frac{1}{3}Rg\Big)\nonumber\\
\intertext{Using that $i_{\vp}(g) = 3 \vp$,}
    &= - \frac{1}{2} d\tau - \frac{1}{2}i_{\vp}(\Ric) + \frac{1}{2}R \vp\nonumber\\
\intertext{Using Lemma \ref{lemma iophi of Ric},}
    &= - \frac{1}{2}d\tau + \frac{1}{2} d\tau - \frac{1}{4} \!\ast\!(\tau^2) + \frac{1}{2} R \vp\nonumber\\
    &= - \frac{1}{4}\!\ast\! (\tau^2) + \frac{1}{2} R \vp.
    \end{align}
Letting $R_{ij} := \Ric_{ij}$ and using~\eqref{equation R = -|T|^2},
\begin{align}
    \big(R^2 + 2R_{ij}T^{il}T_{l}^{\;j}\big) \Vol &= \big(g^{ij}R_{ij}g^{kl}(T^2)_{kl} + 2R^{ij}(T^2)_{ij}\big) \Vol\nonumber\\
\intertext{Applying Lemma \ref{lemma iophi inner product} with $U = \Ric$ and $V = T^2$,}
    &= i_{\vp}(\Ric) \we \ast i_{\vp}(T^2)\nonumber\\
\intertext{Applying Lemma \ref{lemma iophi of Ric} and ~\eqref{equation iphi T2 computation},}
    &= \Big(\!-\!d\tau + \frac{1}{2}\!\ast\!(\tau^2) \Big) \we \ast\Big(\!-\!\frac{1}{4}\!\ast\! (\tau^2) + \frac{1}{2} R \vp\Big)\nonumber\\
&= \frac{1}{4} d\tau \we \tau^2 - \frac{1}{2}R d\tau \we \ast \vp - \frac{1}{8}|\tau^2|^2\Vol + \frac{R}{4}\!\ast\!(\tau^2) \we \ast \vp\nonumber\\
\intertext{Using the symmetry of the inner product on forms and recalling that $|\tau^2|^2 = |\tau|^4$ by~\eqref{equation tau squared},}
&= \frac{1}{12} d(\tau^3) - \frac{1}{2}R d\tau \we \ast \vp - \frac{1}{8}|\tau|^4 \Vol + \frac{R}{4}\vp \we \tau^2\nonumber\\
\intertext{By the fact that there is $\ga \in \Om^3_{27}(M)$ such that $d\tau = \frac{1}{7}|\tau|^2 \vp + \ga$ and that $\ga \we \ast \vp = \ga \we \psi = 0$~(see \cite[(2.14), (4.35)]{BryantRemarks}),}
&= \frac{1}{12} d(\tau^3) - \frac{1}{14}R|\tau|^2 \vp \we \ast \vp - \frac{1}{8}|\tau|^4 \Vol + \frac{R}{4}\vp \we \tau^2\nonumber\\
\intertext{Using~\eqref{equation om214 condition},}
&= \frac{1}{12} d(\tau^3) - \frac{1}{14}R|\tau|^2 \vp \we \ast \vp - \frac{1}{8}|\tau|^4 \Vol- \frac{R}{4}|\tau|^2 \Vol\nonumber\\
\intertext{By ~\eqref{equation R = -|T|^2} and the fact that $\vp \we \ast \vp = 7 \Vol$ (which follows from $\vp_{ljk} \vp_p^{\;jk} = 6 g_{lp}$~\cite[(2.6)]{BryantRemarks}),}
&= \frac{1}{12} d(\tau^3) +R^2 \Vol
\end{align}
Subtracting $R^2 \Vol$ from both sides of the above calculation, we conclude the lemma.
\end{proof}

\section{Proofs of Main Results}
We begin with a simple but useful observation, which will allow us to find signs on certain expressions. 
\begin{lem}\label{lemma negative semidefinite T2}
    Let $T^2$ be the symmetric two-tensor defined by $(T^2)_{ij} := T_{i}^{\;l} T_{lj}$. Then, $T^2$ is negative semidefinite, i.e.\ for each $p \in M$, $T^2|_{p}$ is a negative semidefinite bilinear form. 
\end{lem}
\begin{proof}
    Since $T_{ij}$ is antisymmetric, it follows that $(T^2)^i_j$ is the square of an antisymmetric matrix. This implies that $T^2$ is negative semidefinite. 

    Alternatively, since $T\in \Om^2_{14}$, we may express it explicitly at any point in a $G_2$ frame, as in~\cite[(2.20)]{BryantRemarks}. It follows from a direct computation in these coordinates that $T^2$ is negative semidefinite.
\end{proof}

\begin{proof}[Proof of Theorem \ref{theorem compact negative Ric}]
Suppose $M$ is a closed manifold admitting a closed $G_2$-structure $\vp$ with negative Ricci curvature.

Let $p \in M$, and let $\cE = \{e_1, \dots, e_7\}$ be an orthonormal basis for $T_p M$ which diagonalizes the symmetric two-tensor $(T^2)_{ij} = T_i^{\;l}T_{lj}$. Then,
\begin{equation}\label{equation integrand for integral identity}
    R_{ij}T^{il}T_{l}^{\;j} = \sum_{i=1}^7 R_{ii} T_i^{\;l}T_{li}.
\end{equation}
By assumption, $\Ric< 0$, so $R_{ii}<0$ for each $i$. Also, by Lemma \ref{lemma negative semidefinite T2}, $T_i^{\;l}T_{li} \leq 0$ for each $i$. Since $R_{ii}<0$ and $T_i^{\;l}T_{li}\leq 0$, we have that for each $i$,
\begin{equation}\label{equation integrand for integral identity 0}
    R_{ii} T_i^{\;l}T_{li}\geq 0.
\end{equation}
Applying~\eqref{equation integrand for integral identity 0} to ~\eqref{equation integrand for integral identity}, we have that for each $p \in M$,
\begin{equation}\label{equation integrand for integral identity 2}
    R_{ij}T^{il}T_{l}^{\;j} \geq 0.
\end{equation}
Since $M$ is a closed manifold, Lemma \ref{lemma integral identity Ric T2} tells us that 
\begin{equation}\label{equation integral identity negative Ric}
\int_M R_{ij}T^{il}T_{l}^{\;j} = 0.\end{equation}
Combining~\eqref{equation integral identity negative Ric} with~\eqref{equation integrand for integral identity 2}, we find that for each $p \in M$, 
\begin{equation}\label{equation integrand for integral identity 3}
    R_{ij}T^{il}T_{l}^{\;j} = 0.
\end{equation}
Applying~\eqref{equation integrand for integral identity} and~\eqref{equation integrand for integral identity 0} to~\eqref{equation integrand for integral identity 3} shows that $R_{ii}T_i^{\;l}T_{li}=0$ for each $i$. Since $R_{ii}<0$, we conclude that 
\begin{equation}\label{equation T2 zero}
    T_i^{\;l}T_{li}=0
\end{equation}
for each $i$. Since $\cE$ is a basis which diagonalizes $T^2$ at $p$, we conclude from~\eqref{equation T2 zero} that $T^2 = 0$ for each $p \in M$. Since $T$ is antisymmetric, this implies that $T=0$. Then, $(M, \vp)$ is torsion-free by definition, which implies that $\Ric =0$ for a contradiction. We conclude that there exists no closed manifold $M$ which admits a closed $G_2$-structure with negative Ricci curvature. 
\end{proof}

We now proceed to prove Theorem \ref{theorem Ricci pinched Einstein}. We first give a lemma, Lemma \ref{lemma almost einstein estimate on RT2}, which is needed in order to find a lower bound for the left-hand side of ~\eqref{equation identity of Lemma 2.4}. An alternative lower bound  can be found in~\cite[(4.45)]{BryantRemarks}, but Lemma \ref{lemma almost einstein estimate on RT2} appears to be better suited for our purposes. We note that the inequality we find in Lemma \ref{lemma almost einstein estimate on RT2} is actually an equality when $g_{\vp}$ is Einstein. This follows from letting $k_1 = k_2 = |R|/7$ and using ~\eqref{equation R = -|T|^2}. 

\begin{lem}\label{lemma almost einstein estimate on RT2}
    Let $(M, \vp)$ be a closed $G_2$-structure such that $g_{\vp}$ satisfies~\eqref{equation almost Einstein Ric bound} for $0< k_1 \leq k_2$. Then,
    $$R_{ij}T^{il}T_{l}^{\;j} \geq k_1 |R|.$$
\end{lem}
\begin{proof}
    Let $p \in M$, and let $\cE = \{e_1, \dots, e_7\}$ be an orthonormal basis for $T_p M$ which diagonalizes $\Ric$. It follows from the curvature restriction~\eqref{equation almost Einstein Ric bound} that there exists $\{\la_i\}_{i=1}^7 \subseteq \bR$ such that
\begin{equation}\label{equation eigenvalue bound for almost Einstein}
    -k_2\leq \la_i \leq -k_1<0,
    \end{equation}
and $R_{ij} = \la_i \de_i^j$ in the basis $\cE$ at $p$. Thus, at $p$,
\begin{equation}\label{equation diagonalized Ric traced with T2}
R_{ij}T^{il}T_{l}^{\;j} = \sum_{i=1}^7 \la_i T_{i}^{\;l}T_{li}.
\end{equation}
At $p$ in the basis $\cE$, equation ~\eqref{equation R = -|T|^2} becomes
    \begin{align}\label{equation R = -T2 in basis E}
R = \sum_{i=1}^7 T_{i}^{\;l}T_{li}\leq 0.
\end{align}
By Lemma \ref{lemma negative semidefinite T2}, we have that $T_i^{\;l}T_{li}\leq 0$  for all $i=1, \dots, 7$ at $p$ in the basis $\cE$. Combining this with~\eqref{equation eigenvalue bound for almost Einstein}, ~\eqref{equation diagonalized Ric traced with T2}, and~\eqref{equation R = -T2 in basis E}, we find that
\begin{align}
    R_{ij}T^{il}T_{l}^{\;j} \geq -k_1 \sum_{i=1}^7 T_i^{\;l}T_{li} = k_1|R|.
\end{align}
This concludes the lemma since $p$ is arbitrary.
\end{proof}

We will now lower bound the volume growth of closed $G_2$-structures with negatively pinched Ricci curvature. The key idea is to apply Lemma \ref{lemma almost einstein estimate on RT2} to~\eqref{equation identity of Lemma 2.4} and integrate over balls of arbitrary radius. Note that if you substitute the Einstein metric condition into~\eqref{equation identity of Lemma 2.4}, you obtain the same identity that Bryant used to rule out Einstein closed $G_2$-structures on closed manifolds~\cite[(4.40)]{BryantRemarks}. 

\begin{lem}\label{lemma volume growth Ricci pinched}
    Let $(M, \vp)$ be a closed $G_2$-structure whose metric $g_{\vp}$ is complete and satisfies~\eqref{equation almost Einstein Ric bound} for $0 < k_1 \leq k_2$. Then, for any $\eps>0$ and any $r> \eps$,
\begin{equation}\label{equation volume growth estimate}
 \Vol(B(r)) \geq \Vol(B(\eps)) \exp\left( \frac{6 \sqrt{3} \, k_1^2}{\sqrt{7}\, k_2^{3/2}}(r-\eps)\right).
\end{equation}
\end{lem}

\begin{proof}

Suppose $(M, \vp)$ is a closed $G_2$-structure on a noncompact $M$ such that $g_{\vp}$ is complete and satisfies~\eqref{equation almost Einstein Ric bound} for $0< k_1 \leq k_2$. Then, ~\eqref{equation almost Einstein Ric bound} implies that 
    \begin{equation}\label{equation R bound almost Einstein}
    7k_1 \leq |R|\leq 7k_2.
    \end{equation}

Now, let $p \in M$, and consider the distance function $d_p := d(p, \cdot): M \to \bR$. Recall that the function $d_p$ is Lipschitz on $M$ and is smooth away from the cut locus of $p$. In fact, away from the cut locus of $p$, $|\na d_p|= 1$ and $d_p$ is smooth. In particular, $d_p$ is smooth $\Vol_{g_{\vp}}$\!-almost everywhere on $M$.  

For $r>0$, let $B(r) = \{d_p < r\}$ be the open ball of radius $r$ centered around $p$. The coarea formula for Lipschitz functions on Riemannian manifolds~\cite[Theorem 3.1]{federer1959curvature} applied to $d_p$ implies that $r \mapsto\Vol(B(r))$ is absolutely continuous for $r>0$ and implies that there is a conull set $\cS \subset (0,\infty)$ such that for all $r \in \cS$, 
\begin{equation}\label{equation coarea formula implication}\frac{d}{dr}\Vol(B(r)) = \cH^6(\{d_p = r\}).\end{equation}
In addition, we may find the conull set $\cS \subseteq (0, \infty)$ such that $B(r)$ is a set of finite perimeter for all $r \in \cS$. Note that this follows in particular from the coarea formula for BV functions on metric measure spaces~\cite[Corollary 4.4]{miranda2003functions}. 

Now, let $r \in \cS$. Integrating~\eqref{equation identity of Lemma 2.4} on $B(r)$,
\begin{align}\label{equation integral on B(r)}
      \int_{B(r)} 24 R_{ij}T^{il}T_{l}^{\;j}
    &= \int_{B(r)} d(\tau^3)\nonumber\\
\intertext{Since $B(r)$ is a set of finite perimeter for $r \in \cS$, we may consider the associated $7$-current $\llbracket B(r)\rrbracket$, with orientation induced by that of $(M, \vp)$, and its boundary $\dd \llbracket B(r) \rrbracket$,  which is a locally integer rectifiable $6$-current. Now, let $\eta \in C^{\infty}_c(M, \bR)$ so that $0\leq \eta \leq 1$ on $M$ and $\eta \equiv 1$ on $\spt \llbracket B(r)\rrbracket$. Since $\eta \tau^3$ is compactly supported on $M$,}
    &= \dd\llbracket B(r)\rrbracket (\eta \tau^3)\nonumber\\
\intertext{We now bound $\dd\llbracket B(r)\rrbracket(\eta \tau^3)$ by the norm of $\eta \tau^3$ as a $6$-form on $M$ together with the mass of $\dd\llbracket B(r)\rrbracket$. Here, the mass $\bf{M}(\dd \llbracket B(r) \rrbracket)$ is defined using the inner product on forms~\cite[Section 6.1]{simon2014introduction}.}
    &\leq \Big(\sup_{M}|\eta \tau^3| \Big)\bf{M}(\dd \llbracket B(r)\rrbracket)\nonumber\\
\intertext{We note that $\sup_M |\eta \tau^3|$ is finite in this setting. Indeed, using that $\eta \leq 1$ and applying~\eqref{equation tau cubed} with $\tau \in \Om^2_{14}$ so that $|\tau^3|^2 \leq \frac{2}{3}|\tau|^6$,}
    &\leq \sqrt{\frac{2}{3}} \Big(\sup_M|\tau|^3\Big)\bf{M}(\dd \llbracket B(r)\rrbracket)\nonumber\\
\intertext{Applying ~\eqref{equation R = -|T|^2},}
    &= \frac{4}{\sqrt{3}} \Big(\sup_M |R|^{3/2}\Big) \bf{M}(\dd \llbracket B(r)\rrbracket)\nonumber\\
\intertext{Since $B(r)$ is a set of finite perimeter for $r \in \cS$, de Giorgi's structure theory for sets of finite perimeter says that $\mu_{\dd\llbracket B(r)\rrbracket} = \cH^6 \lfloor \dd^* \!B(r)$, where $\dd^* \!B(r)$ is the reduced boundary of $B(r)$ and $\mu_{\dd\llbracket B(r)\rrbracket}$ is the mass measure associated to $\dd\llbracket B(r)\rrbracket$. We refer to ~\cite{volkmann2010regularity} and~\cite[Definition 2.3(4)]{BackusGGTRiemannian23} for these facts in the Riemannian setting. Thus,}
    &= \frac{4}{\sqrt{3}} \Big(\sup_M |R|^{3/2}\Big) \cH^{6}(\dd^* \!B(r))\nonumber\\
\intertext{Since $\dd^* \!B(r) \subseteq \dd B(r) \subseteq \{d_p=r\}$ (see, for instance,~\cite[Remark 2.8]{LiWangGenericRegularity20}),}
    &\leq \frac{4}{\sqrt{3}} \Big(\sup_M |R|^{3/2}\Big) \cH^{6}(\{d_p = r\})\nonumber\\
\intertext{Applying~\eqref{equation R bound almost Einstein} to bound $\sup_M |R|^{3/2}$ and using ~\eqref{equation coarea formula implication} for $r \in \cS$,}
    &\leq \frac{4}{\sqrt{3}} (7k_2)^{3/2} \frac{d}{dr}\Vol\big(B(r)\big).
\end{align}
Thus, we find that for all $r \in \cS$, 
\begin{align}\label{equation Ricci pinched differential volume estimate}
    \frac{4}{\sqrt{3}} (7k_2)^{3/2} \frac{d}{dr}\Vol\big(B(r)\big) &\geq \int_{B(r)} \!24 R_{ij}T^{il}T_{l}^{\;j}\nonumber\\
\intertext{Applying Lemma \ref{lemma almost einstein estimate on RT2},}
    &\geq 24 k_1\int_{B(r)}\!|R|\nonumber\\
\intertext{Applying~\eqref{equation R bound almost Einstein},}
    &\geq 168 k_1^2 \Vol(B(r)).
\end{align}

Since $\Vol(B(r))$ is absolutely continuous for $r>0$, $r \mapsto\log\big(\!\Vol(B(r))\big)$ is also absolutely continuous for $r>0$ and in fact differentiable for $r \in \cS$. We then find from ~\eqref{equation Ricci pinched differential volume estimate} that for $r \in \cS$,
\begin{equation}\label{equation Ricci pinched differential volume estimate log}
    \frac{d}{dr}\Big(\log(\Vol(B(r))\Big)\geq \frac{6 \sqrt{3} \, k_1^2}{\sqrt{7}\, k_2^{3/2}}.
\end{equation}
Let $\eps>0$. Since $\cS$ is conull in $\bR$, we may integrate ~\eqref{equation Ricci pinched differential volume estimate log} from $\eps$ to any $r>\eps$ to find that
\begin{equation}\label{equation Ricci pinched integrated volume estimate}
    \Vol(B(r)) \geq \Vol(B(\eps)) \exp\left(\frac{6 \sqrt{3} \, k_1^2}{\sqrt{7}\, k_2^{3/2}}(r-\eps)\right).
\end{equation}
\end{proof}
\begin{rmk}
    There are strong restrictions on any closed $G_2$-structure satisfying~\eqref{equation volume growth estimate} with equality. In the proof of Lemma \ref{lemma volume growth Ricci pinched}, we applied inequality~\eqref{equation tau cubed} to the torsion, and equality in~\eqref{equation tau cubed} implies that $(M, \vp)$ has ``special torsion of negative type''~\cite[\S 3.4]{BallQuadratic}. Under the extra assumption that the closed $G_2$-structure is $\la$-quadratic, having special torsion of negative type implies that the metric is incomplete~\cite[\S 1.1.3]{BallQuadratic}.  
\end{rmk}

We now prove Theorem \ref{theorem Ricci pinched Einstein} using Lemma \ref{lemma volume growth Ricci pinched}.

\begin{proof}[Proof of Theorem \ref{theorem Ricci pinched Einstein}]

We now define $K_2 := \frac{k_2}{6}$, so that by~\eqref{equation almost Einstein Ric bound}, 
$$\Ric \geq -6K_2 g,$$
with $g = g_{\vp}$. Applying the Bishop--Gromov inequality for any $r \geq \eps>0$, we find that
\begin{equation}\label{equation Bishop-Gromov}
\Vol(B(r)) \leq \frac{\Vol(B(\eps))}{\Vol_{-K_2}(B(\eps))} \Vol_{-K_2}(B(r)),
\end{equation}
where $\Vol_{-K_2}(B(r))$ and $\Vol_{-K_2}(B(\eps))$ denote the volumes of radius $r$- and $\eps$-balls, respectively, in a complete, simply-connected $7$-manifold $N_{-K_2}$ of constant sectional curvature $-K_2<0$. 

Now, the volume of a ball of radius $r>0$ in $N_{-K_2}$ is given by
\begin{equation}\label{equation hyperbolic space volume growth}
    \Vol_{-K_2} (B(r)) = \Vol_{\bR^7}(\bS^6) \int_0^r \Big( \frac{\sinh(\sqrt{K_2}s)}{\sqrt{K_2}}\Big)^6 ds,
\end{equation}
where $\Vol_{\bR^7}(\bS^6)$ is the surface area of the unit $6$-sphere in Euclidean $\bR^7$.

Combining~\eqref{equation Bishop-Gromov} with~\eqref{equation hyperbolic space volume growth},
\begin{align}\label{equation volume estimate from Bishop-Gromov}
    &\Vol(B(r))\leq \Vol(B(\eps)) \frac{\int_0^r \sinh^6(\sqrt{K_2}s)\, ds}{\int_0^{\eps} \sinh^6(\sqrt{K_2}s)\, ds}\nonumber\\
\intertext{Computing these integrals and dividing by common factors,}
    &= \Vol(B(\eps)) \frac{\sinh(6\sqrt{K_2}r) - 9\sinh(4\sqrt{K_2}r)+45\sinh(2\sqrt{K_2}r) - 60 \sqrt{K_2}r}{\sinh(6\sqrt{K_2}\eps) - 9\sinh(4\sqrt{K_2}\eps)+45\sinh(2\sqrt{K_2}\eps) - 60 \sqrt{K_2}\eps}.
\end{align}

Combining the lower bound on volume growth found in Lemma \ref{lemma volume growth Ricci pinched} with the upper bound in~\eqref{equation volume estimate from Bishop-Gromov}, we find that for any $0 < \eps \leq r$,
\begin{align*}
    \exp\Bigg(&\frac{6 \sqrt{3}\,k_1^2}{\sqrt{7}\, k_2^{3/2}}(r-\eps)\Bigg)\nonumber\\
    &\leq \frac{\sinh(6\sqrt{K_2}r) - 9\sinh(4\sqrt{K_2}r)+45\sinh(2\sqrt{K_2}r) - 60 \sqrt{K_2}r}{\sinh(6\sqrt{K_2}\eps) - 9\sinh(4\sqrt{K_2}\eps)+45\sinh(2\sqrt{K_2}\eps) - 60 \sqrt{K_2}\eps},
\end{align*}
which implies that the following function is nondecreasing for $r>0$:
\begin{equation}\label{equation Ricci pinched volume contradiction estimate}
    r \mapsto \frac{\sinh(6\sqrt{K_2}r) - 9\sinh(4\sqrt{K_2}r)+45\sinh(2\sqrt{K_2}r) - 60 \sqrt{K_2}r}{\exp\Big( \frac{6 \sqrt{3}\,k_1^2}{\sqrt{7}\,k_2^{3/2}}r\Big)}.
\end{equation}
The numerator of the expression in~\eqref{equation Ricci pinched volume contradiction estimate} grows, up to a constant, like $C\exp(6\sqrt{K_2}r)$. Thus, if
\begin{equation}\label{equation Ricci pinched contradiction condition}
\frac{6 \sqrt{3}\,k_1^2}{\sqrt{7}\, k_2^{3/2}}>6 \sqrt{K_2} = \sqrt{6k_2},
\end{equation}
this would contradict the fact that the function in ~\eqref{equation Ricci pinched volume contradiction estimate} is nondecreasing.
The inequality~\eqref{equation Ricci pinched contradiction condition} is equivalent to 
\begin{equation}\label{equation k1/k2 condition}
\frac{k_1}{k_2} > \left(\frac{7}{18}\right)^{\frac{1}{4}} \cong .7897.
\end{equation}
We conclude that there does not exist any complete closed $G_2$-structure whose Ricci curvature satisfies~\eqref{equation almost Einstein Ric bound} with $k_1, k_2$ subject to inequality~\eqref{equation k1/k2 condition}. This finishes the proof of Theorem \ref{theorem Ricci pinched Einstein}.
\end{proof}

We now prove that Theorem \ref{theorem Ricci pinched Einstein} implies Corollary \ref{corollary complete noncompact Einstein}.

\begin{proof}[Proof of Corollary \ref{corollary complete noncompact Einstein}]
    Suppose $(M, \vp)$ is a closed $G_2$-structure such that the metric $g_{\vp}$ is complete and Einstein. We may assume $M$ is noncompact, since the compact case of Corollary \ref{corollary complete noncompact Einstein} has already been proven by Cleyton--Ivanov and Bryant~\cite{CleytonIvanovGeometryClosed, BryantRemarks}. 
    
    Since $M$ is an Einstein $7$-manifold, the scalar curvature $R$ is constant and $\Ric = \frac{R}{7}g_{\vp}$. By ~\eqref{equation R = -|T|^2}, $R = -|T|^2 \leq 0$. If $R<0$, then by applying Theorem \ref{theorem Ricci pinched Einstein} with $k_1 = k_2 = \frac{|R|}{7}$, we find a contradiction. Thus, $R=0$ on $M$, and~\eqref{equation R = -|T|^2} implies that $T=0$, i.e.\ $(M, \vp)$ is torsion-free.
\end{proof}

The proof of Theorem \ref{theorem Ricci pinched Einstein} also gives an estimate for how inextensible geodesics are in closed $G_2$-structures satisfying~\eqref{equation almost Einstein Ric bound} and ~\eqref{equation pinching of k1 and k2}. 

\begin{proof}[Proof of Theorem \ref{theorem incomplete Einstein}]
Let $(M, \vp)$ be a closed $G_2$-structure with a possibly incomplete metric $g_{\vp}$ satisfying~\eqref{equation almost Einstein Ric bound} and~\eqref{equation pinching of k1 and k2}. By scaling, we may assume without loss of generality that $k_2 = 1$. 

Let $p \in M$. Suppose that for some $R>0$, the exponential map $\exp_p$ is defined for all $v \in T_p M$ with $|v|_{g_{\vp}} \leq R$. The arguments of Lemma \ref{lemma volume growth Ricci pinched} hold in this setting for $r \leq R$, meaning that Lemma \ref{lemma volume growth Ricci pinched} is true for $\eps < r \leq R$ with balls $B(r)$ based at $p$. Using this fact, the argument of Theorem \ref{theorem incomplete Einstein} also goes through for $r \leq R$ up through~\eqref{equation Ricci pinched volume contradiction estimate}. In other words, setting $k_2 = 1$ in~\eqref{equation Ricci pinched volume contradiction estimate}, the following function is nondecreasing for $0<r\leq R$:
\begin{equation}\label{equation nondecreasing k2=1}
\frac{\sinh(6\sqrt{1/6}\,r) - 9\sinh(4\sqrt{1/6}\,r)+45\sinh(2\sqrt{1/6}\,r) - 60 \sqrt{1/6}\,r}{\exp\Big( \frac{6 \sqrt{3} k_1^2}{\sqrt{7}}r\Big)}.
\end{equation}
Given~\eqref{equation almost Einstein Ric bound} and~\eqref{equation pinching of k1 and k2}, it is easy to see that there is $C(k_1)>0$ such that the first derivative of the above function, ~\eqref{equation nondecreasing k2=1}, has a sign change at $r=C(k_1)$. If $R>C(k_1)$, the sign change at $r=C(k_1)$ contradicts the fact that~\eqref{equation nondecreasing k2=1} is nondecreasing for $r \leq R$. This implies that we may find $v \in T_p M$ with $C(k_1) < |v|_{g_{\vp}} < R$ such that $\exp_p (v)$ is not defined. The main statement of Theorem \ref{theorem incomplete Einstein} follows.

Now, if $g_{\vp}$ is an Einstein metric, then $k_1 = k_2 = 1$ in the scaling we have chosen. Setting $k_1 = 1$ in~\eqref{equation nondecreasing k2=1}, we find that this function has a sign change at $r=C \approx 2.05$. 
\end{proof}


\bibliographystyle{amsplain}
\bibliography{bibliography}

\end{document}